\documentclass[11pt]{article}

\usepackage{geometry}
\usepackage{amsmath, amsthm}
\usepackage{amssymb}

\newtheorem{thm}{Theorem}
\newtheorem{lem}[thm]{Lemma}

\theoremstyle{remark}
\newtheorem{rem}[thm]{Remark}

\def\ve{\varepsilon}

\def\Pr{{\mathbb P}}
\def\Ex{{\mathbb E}}
\def\er{{\mathbb R}}

\def\ind{\mathbf{1}}

\def\diag{\mathrm{diag}}
\def\caln{\mathcal{N}}

\title{Royen's proof of the Gaussian correlation inequality}
\author{Rafa{\l} Lata{\l}a and Dariusz Matlak}
\date{}

\begin{document}

\maketitle

\begin{abstract}
We present in detail Thomas Royen's proof of the Gaussian correlation inequality which states that
$\mu(K\cap L)\geq \mu(K)\mu(L)$ for
any centered Gaussian measure $\mu$ on $\er^d$ and symmetric convex sets $K,L$ in $\er^d$.
\end{abstract}

\section{Introduction}
The aim of this note is to present in a self contained way the beautiful proof of the Gaussian correlation 
inequality, due to Thomas Royen \cite{Ro}. Although the method is rather simple and elementary, 
we found the original paper not too easy to follow. One of the reasons behind it is that in \cite{Ro} the correlation
inequality was established for more general class of probability measures. Moreover, the author assumed that the
reader is familiar with properties of certain distributions and may justify some calculations by herself/himself. 
We decided to reorganize a bit Royen's proof, restrict it only to the Gaussian case and add some missing details.
We hope that this way a wider readership may appreciate the remarkable result of Royen.   


The statement of the Gaussian correlation inequality is as follows.

\begin{thm}
\label{thm:gcc}
For any closed symmetric sets $K,L$ in $\er^d$ and any centered Gaussian measure $\mu$ on $\er^d$ we have
\begin{equation}
\label{eq:gcc}
\mu(K\cap L)\geq \mu(K)\mu(L).
\end{equation}
\end{thm}

For $d=2$ the result was proved by Pitt \cite{Pi}. In the case when one of the sets $K,L$ is a symmetric strip 
(which corresponds to $\min\{n_1,n_2\}=1$ in Theorem \ref{thm:gcc2} below) inequality \eqref{eq:gcc} was established 
independently by Khatri \cite{Kh} and \v{S}id\'ak \cite{Si}. 
Harg\'e \cite{Ha} generalized the Khatri-\v{S}idak result to the case when one of the sets is a symmetric ellipsoid.
Some other partial results may be found in papers of Borell \cite{Bo} and Schechtman, Schlumprecht and Zinn \cite{SSZ}.

Up to our best knowledge Thomas Royen was the first to present a complete proof of the Gaussian correlation inequality.
Some other recent attempts may be found in \cite{Me} and \cite{Qi}, however both papers are very long and difficult to check. 
The first version of \cite{Me}, placed on the arxiv before Royen's paper, contained a fundamental mistake
(Lemma 6.3 there was wrong).

Since any symmetric closed set is a countable intersection of symmetric strips, it is enough to show \eqref{eq:gcc}
in the case when
\[
K=\{x\in \er^d\colon\ \forall_{1\leq i\leq n_1}\ |\langle x,v_i\rangle|\leq t_i\}\quad \mbox{and} \quad
L=\{x\in\er^d \colon\ \forall_{n_1+1\leq i\leq n_1+n_2}\ |\langle x,v_i\rangle|\leq t_i\},
\]
where $v_i$ are vectors in $\er^d$ and $t_i$ nonnegative numbers. If we set $n=n_1+n_2$, $X_i:=\langle v_i,G\rangle$,
where $G$ is the Gaussian random vector distributed according to $\mu$, we obtain the following equivalent form of 
Theorem \ref{thm:gcc}. 

\begin{thm}
\label{thm:gcc2}
Let $n=n_1+n_2$ and $X$ be an $n$-dimensional centered Gaussian vector. 
Then for any $t_1,\ldots,t_n>0$,
\begin{align*}
\Pr(|X_1|\leq t_1,\ldots,&|X_{n}|\leq t_n) 
\\
&\geq \Pr(|X_1|\leq t_1,\ldots,|X_{n_1}|\leq t_{n_1})\Pr(|X_{n_1+1}|\leq t_{n_1+1},\ldots,|X_{n}|\leq t_{n}).
\end{align*}
\end{thm}

\begin{rem}
i) The standard approximation argument shows that the Gaussian correlation inequality holds for centered Gaussian measures
on separable Banach spaces. \\
ii) Thomas Royen established Theorem \ref{thm:gcc2} for more general class of random vectors $X$ such
that $X^2=(X_1^2,\ldots,X_n^2)$ has an $n$-variate gamma distribution (see \cite{Ro} for details). 
\end{rem}


{\bf Notation.} By ${\mathcal N}(0,C)$ we denote the centered Gaussian measure with the covariance matrix $C$. 
We write $M_{n\times m}$ for a set of $n\times m$ matrices and $|A|$ for the determinant of a square matrix $A$.
For a matrix $A=(a_{ij})_{i,j\leq n}$ and $J\subset [n];=\{1,\ldots,n\}$ by $A_J$ we denote the square matrix 
$(a_{ij})_{i,j\in J}$ and by $|J|$ the cardinality of $J$.

\section{Proof of Theorem \ref{thm:gcc2}}

Without loss of generality we may and will assume that the covariance matrix $C$ of $X$ is nondegenerate 
(i.e. strictly positively defined). 
We may write $C$ as
\[
C=\left(\begin{array}{ll}
C_{11}& C_{12}
\\
C_{21}& C_{22}
\end{array}\right),
\]
where $C_{ij}$ is the $n_i\times n_j$ matrix. Let
\[
C(\tau):=\left(
\begin{array}{ll}
C_{11}& \tau C_{12}
\\
\tau C_{21}& C_{22}
\end{array}\right),
\quad 0\leq \tau\leq 1.
\]
Set $Z_i(\tau):=\frac{1}{2}X_i(\tau)^2$, $1\leq i\leq n$, where $X(\tau)\sim\caln(0,C(\tau))$.

We may restate the assertion as
\[
\Pr(Z_1(1)\leq s_1,\ldots,Z_n(1)\leq s_n)\geq \Pr(Z_1(0)\leq s_1,\ldots,Z_n(0)\leq s_n),
\]
where $s_1=\frac{1}{2}t_i^2$.
Therefore it is enough to show that the function
\[
\tau\mapsto \Pr(Z_1(\tau)\leq s_1,\ldots,Z_n(\tau)\leq s_n) \mbox{ is nondecreasing on }[0,1]. 
\]

Let $f(x,\tau)$ denote the density of the random vector $Z(\tau)$ and $K=[0,s_1]\times\cdots\times [0,s_n]$. We have
\[
\frac{\partial}{\partial\tau}\Pr(Z_1(\tau)\leq s_1,\ldots,Z_n(\tau)\leq s_n)
=\frac{\partial}{\partial\tau}\int_K f(x,\tau) dx=\int_K \frac{\partial}{\partial\tau}f(x,\tau) dx,
\]
where the last equation follows by Lemma \ref{lem:diff} applied to $\lambda_1=\ldots=\lambda_n=0$. 
Therefore it is enough to show that
$\int_K \frac{\partial}{\partial\tau}f(x,\tau)\geq 0$.

To this end we will compute the Laplace transform of $\frac{\partial}{\partial\tau}f(x,\tau)$. By Lemma \ref{lem:diff}, applied to
$K=[0,\infty)^n$,
we have for any $\lambda_1\ldots,\lambda_n\geq 0$,
\[
\int_{[0,\infty)^n}e^{-\sum_{i=1}^n\lambda_i x_i}\frac{\partial}{\partial\tau}f(x,\tau)dx
=\frac{\partial}{\partial\tau}\int_{[0,\infty)^n}e^{-\sum_{i=1}^n\lambda_i x_i}f(x,\tau)dx.
\]
However by Lemma \ref{lem:Lap} we have
\[
\int_{[0,\infty)^n}e^{-\sum_{i=1}^n\lambda_i x_i}f(x,\tau)dx=
\Ex\exp\left(-\frac{1}{2}\sum_{k=1}^n\lambda_k X_k^2(\tau)\right)=|I+\Lambda C(\tau)|^{-1/2},
\]
where $\Lambda=\diag(\lambda_1,\ldots,\lambda_n)$.

Formula \eqref{eq:det1} below yields
\[
|I+\Lambda C(\tau)|=1+\sum_{\emptyset\neq J\subset [n]}|(\Lambda C(\tau))_J|
=1+\sum_{\emptyset\neq J\subset [n]}|C(\tau)_J|\prod_{j\in J}\lambda_j.
\]
Fix $\emptyset\neq J\subset [n]$. Then $J=J_1\cup J_2$, where  $J_1:=[n_1]\cap J$, $J_2:=J\setminus [n_1]$ and
$C(\tau)_J=\left(\begin{array}{ll} C_{J_1}& \tau C_{J_1 J_2}\\ \tau C_{J_2 J_1}& C_{J_2} \end{array}\right)$.
If $J_1=\emptyset$ or $J_2=\emptyset$ then $C(\tau)_J=C_J$, otherwise by \eqref{eq:det2} we get
\begin{align*}
|C(\tau)_J|
&=|C_{J_1}||C_{J_2}|\left|I_{|J_1|}-\tau^{2}C_{J_1}^{-1/2}C_{J_1J_2}C_{J_2}^{-1}C_{J_2J_1}C_{J_1}^{-1/2}\right|
\\
&=|C_{J_1}||C_{J_2}|\prod_{i=1}^{|J_1|}(1-\tau^2 \mu_{J_1,J_2}(i)),
\end{align*}
where $\mu_{J_1,J_2}(i)$, $1\leq i\leq |J_1|$ denote the eigenvalues of 
$C_{J_1}^{-1/2}C_{J_1J_2}C_{J_2}^{-1}C_{J_2J_1}C_{J_1}^{-1/2}$ (by \eqref{eq:det2a} they belong to $[0,1]$).
Thus for any $\emptyset\neq J\subset [n]$ and $\tau\in [0,1]$ we have
\[
a_J(\tau):=-\frac{\partial}{\partial \tau}|C(\tau)_J|\geq 0.
\]
Therefore
\begin{align*}
\frac{\partial}{\partial \tau}|I+\Lambda C(\tau)|^{-1/2}
&=-\frac{1}{2}|I+\Lambda C(\tau)|^{-3/2}\sum_{\emptyset\neq J\subset [n]}\frac{\partial}{\partial \tau}|C(\tau)_J||\Lambda_J|
\\
&=\frac{1}{2} |I+\Lambda C(\tau)|^{-3/2}\sum_{\emptyset\neq J\subset [n]}a_J(\tau)\prod_{j\in J}\lambda_j.
\end{align*}
We have thus shown  that
\[
\int_{[0,\infty)^n}e^{-\sum_{i=1}^n\lambda_i x_i}\frac{\partial}{\partial\tau}f(x,\tau)dx
= \sum_{\emptyset\neq J\subset [n]}\frac{1}{2}a_J(\tau)|I+\Lambda C(\tau)|^{-3/2}\prod_{j\in J}\lambda_j.
\]

Let $h_{\tau}:=h_{3,C(\tau)}$ be the density function on $(0,\infty)^n$ defined by \eqref{eq:defhkC}. 
By Lemmas \ref{lem:LaplhkC}
and \ref{lem:proph} iii) we know that
\[
|I+\Lambda C(\tau)|^{-3/2}\prod_{j\in J}\lambda_j=\int_{(0,\infty)^n}e^{-\sum_{i=1}^n\lambda_i x_i}
\frac{\partial^{|J|}}{\partial x_J}h_{\tau}.
\]
This shows that
\[
\frac{\partial}{\partial\tau}f(x,\tau)=
\sum_{\emptyset\neq J\subset [n]}\frac{1}{2}a_J(\tau)\frac{\partial^{|J|}}{\partial x_J}h_{\tau}(x).
\]

Finally recall that $a_J(\tau)\geq 0$ and observe that by Lemma \ref{lem:proph} ii),
\[
\lim_{x_i\to 0+}\frac{\partial^{|I|}}{\partial x_{I}}h_{\tau}(x)\quad \mbox{ for }i\notin I\subset [n],
\]
thus
\[
\int_K \frac{\partial^{|J|}}{\partial x_J}h_{\tau}(x)dx=
\int_{\prod_{j\in J^c}[0,t_j]}h_{\tau}(t_J,x_{J^c})dx_{J^c}\geq 0,
\] 
where $J^c=[n]\setminus J$ and $y=(t_J,x_{J^c})$ if $y_i=t_i$ for $i\in J$ and $y_i=x_i$ for $i\in J^c$. \hfill $\qed$

\section{Auxiliary Lemmas}

\begin{lem}
\label{lem:Lap}
Let $X$ be an $n$ dimensional centered Gaussian vector with the covariance matrix $C$. Then for any 
$\lambda_1,\ldots,\lambda_n\geq 0$ we have 
\[
\Ex\exp\left(-\sum_{i=1}^n\lambda_iX_i^2\right)=|I_n+2\Lambda C|^{-1/2},
\]
where $\Lambda:=\diag(\lambda_1,\ldots,\lambda_n)$.
\end{lem}

\begin{proof}
Let $A$ be a symmetric positively defined matrix. Then $A=UDU^T$ for some $U\in O(n)$ and 
$D=\diag(d_1,d_2,\ldots,d_n)$. Hence
\[
\int_{\er^n}\exp(-\langle Ax,x\rangle)dx=\int_{\er^n}\exp(-\langle Dx,x\rangle)dx
=\prod_{k=1}^n\sqrt{\frac{\pi}{d_k}}=\pi^{n/2}|D|^{-1/2}=\pi^{n/2}|A|^{-1/2}.
\]
Therefore for a canonical Gaussian vector $Y\sim\caln(0,I_n)$ and a symmetric matrix $B$
such that $2B<I_n$ we have
\begin{align*}
\Ex\exp(\langle BY,Y\rangle)
&=(2\pi)^{-n/2}\int_{\er^n}\exp\left(-\left\langle \left(\frac{1}{2}I_n-B\right)x,x\right\rangle\right)dx
=2^{-n/2}\left|\frac{1}{2}I_n-B\right|^{-1/2}
\\
&=|I_n-2B|^{-1/2}.
\end{align*}

We may represent $X\sim \caln(0,C)$ as $X\sim AY$ with $Y\sim{\cal N}(0,I_n)$ and $C=AA^T$. Thus 
\begin{align*}
\Ex \exp\left(-\sum_{i=1}^n\lambda_i X_i^2\right)
&=\Ex \exp(-\langle\Lambda X,X\rangle)
=\Ex \exp(-\langle\Lambda AY,AY\rangle)
=\Ex \exp(-\langle A^T\Lambda AY,Y\rangle)
\\
&=|I_n+2A^T\Lambda A|^{-1/2}=|I_n+2\Lambda C|^{-1/2},
\end{align*}
where to get the last equality we used the fact that $|I_n+A_1A_2|=|I_n+A_2A_1|$ for $A_1,A_n\in M_{n\times n}$.
\end{proof}

\begin{lem}
i) For any matrix $A\in M_{n\times n}$,
\begin{equation}
\label{eq:det1}
|I_n+A|=1+\sum_{\emptyset\neq J\subset [n]}|A_J|.
\end{equation}
ii) Suppose that $n=n_1+n_2$ and $A\in M_{n\times n}$has the block representation 
$A=\left(\begin{array}{ll} A_{11}& A_{12}\\A_{21}& A_{22} \end{array}\right)$, 
where $A_{ij}\in M_{n_i\times n_j}$ and $A_{11}$, $A_{22}$ are invertible.
Then 
\begin{equation}
\label{eq:det2}
|A|=|A_{11}||A_{22}|\left|I_{n_1}-A_{11}^{-1/2}A_{12}A_{22}^{-1}A_{21}A_{11}^{-1/2}\right|.
\end{equation}
Moreover, if $A$ is symmetric and positively defined then
\begin{equation}
\label{eq:det2a}
0\leq A_{11}^{-1/2}A_{12}A_{22}^{-1}A_{21}A_{11}^{-1/2}\leq I_{n_1}.
\end{equation}
\end{lem}

\begin{proof}
i) This formula may be verified in several ways -- e.g. by induction on $n$ or by using the Leibniz formula for the
determinant.

ii) We have
\[
\left(\begin{array}{cc} A_{11}& A_{12} \\A_{21}& A_{22} \end{array}\right)
=\left(\begin{array}{cc} A_{11}^{1/2}& 0 \\0 & A_{22}^{1/2} \end{array}\right)
\left(\begin{array}{cc} I_{n_1}& A_{11}^{-1/2}A_{12}A_{22}^{-1/2} 
\\A_{22}^{-1/2}A_{21}A_{11}^{-1/2}& I_{n_2} \end{array}\right)
\left(\begin{array}{cc} A_{11}^{1/2}& 0 \\0& A_{22}^{1/2} \end{array}\right)
\]
and
\begin{align*}
\left|\left(\begin{array}{cc} I_{n_1}& A_{11}^{-1/2}A_{12}A_{22}^{-1/2} 
\\A_{22}^{-1/2}A_{21}A_{11}^{-1/2}& I_{n_2} \end{array}\right)\right|
&=
\left|\left(\begin{array}{cc} I_{n_1}-A_{11}^{-1/2}A_{12}A_{22}^{-1}A_{21}A_{11}^{-1/2}& 0 
\\A_{22}^{-1/2}A_{21}A_{11}^{-1/2}& I_{n_2} \end{array}\right)\right|
\\
&=\left|I_{n_1}-A_{11}^{-1/2}A_{12}A_{22}^{-1}A_{21}A_{11}^{-1/2}\right|.
\end{align*}

To show the last part of the statement notice that 
$A_{11}^{-1/2}A_{12}A_{22}^{-1}A_{21}A_{11}^{-1/2}=B^TB\geq 0$, where $B:=A_{22}^{-1/2}A_{21}A_{11}^{-1/2}$.
If $A$ is positively defined then for any $t\in\er$, $x\in \er^{n_1}$ and $y\in \er^{n_2}$
we have $t^2\langle A_{11}x,x\rangle+2t\langle A_{21}x,y\rangle+\langle A_{22}y,y\rangle\geq 0$.
This implies $\langle A_{21}x,y\rangle^2\leq \langle A_{11}x,x\rangle\langle A_{22}y,y\rangle$.
Replacing $x$ by $A_{11}^{-1/2}x$ and $y$ by $A_{22}^{-1/2}y$ we get
$\langle Bx,y\rangle^2\leq |x|^2|y|^2$. Choosing $y=Bx$ we get $\langle B^TBx,x\rangle\leq |x|^2$,
i.e. $B^TB\leq I_{n_1}$.
\end{proof}

\begin{lem}
\label{lem:diff}
Let  $f(x,\tau)$ be the density of the random vector $Z(\tau)$ defined above. Then for any Borel set $K$ in $[0,\infty)^n$
and any $\lambda_1,\ldots,\lambda_n\geq 0$,
\[
\int_{K}e^{-\sum_{i=1}^n\lambda_i x_i}\frac{\partial}{\partial\tau}f(x,\tau)dx
=\frac{\partial}{\partial\tau}\int_{K}e^{-\sum_{i=1}^n\lambda_i x_i}f(x,\tau)dx.
\]
\end{lem}

\begin{proof}
The matrix $C$ is nondegenerate,  therefore matrices $C_{11}$ and $C_{22}$ are nondegerate and $C(\tau)$ is 
nondegenerate for any $\tau\in [0,1]$. Random vector $X(\tau)\sim {\mathcal N}(0,C(\tau))$ has the density
$|C(\tau)|^{-1/2}(2\pi)^{-n/2}\exp(-\frac{1}{2}\langle C(\tau)^{-1}x,x\rangle)$. 
Standard calculation shows that $Z(\tau)$ has the density
\[
f(x,\tau)=|C(\tau)|^{-1/2}(4\pi)^{-n/2}\frac{1}{\sqrt{x_1\cdots x_n}}
\sum_{\ve\in \{-1,1\}^n}e^{-\langle C(\tau)^{-1}\sqrt{x}_\ve,\sqrt{x}_\ve\rangle}\ind_{(0,\infty)^n}(x),
\]
where for $\ve\in \{-1,1\}^n$ and $x\in(0,\infty)^n$ we set $\sqrt{x}_{\ve}:=(\ve_i \sqrt{x_i})_i$.  

The function $\tau\mapsto |C(\tau)|^{-1/2}$ is smooth on $[0,1]$, in particular 
\[
\sup_{\tau\in [0,1]}|C(\tau)|^{-1/2}+\sup_{\tau\in [0,1]}\frac{\partial}{\partial \tau}|C(\tau)|^{-1/2}=:M<\infty.
\] 
Since $C(\tau)=\tau C(1)+(1-\tau)C(0)$ we have $\frac{\partial}{\partial \tau}C(\tau)=C(1)-C(0)$ and
\[
\frac{\partial}{\partial \tau}e^{-\langle C(\tau)^{-1}\sqrt{x}_\ve,\sqrt{x}_\ve\rangle}
=-\langle C(\tau)^{-1}(C(1)-C(0))C(\tau)^{-1}\sqrt{x}_\ve,\sqrt{x}_\ve\rangle
e^{-\langle C(\tau)^{-1}\sqrt{x}_\ve,\sqrt{x}_\ve\rangle} .
\]
The continuity of the function $\tau\mapsto C(\tau)$ gives 
\[
\langle C(\tau)^{-1}\sqrt{x}_\ve,\sqrt{x}_\ve\rangle\geq a\langle \sqrt{x}_\ve,\sqrt{x}_\ve\rangle=a\sum_{i=1}^n|x_i|
\]
 and 
\[
\langle C(\tau)^{-1}(C(1)-C(0))C(\tau)^{-1}\sqrt{x}_\ve,\sqrt{x}_\ve\rangle
\leq b\langle \sqrt{x}_\ve,\sqrt{x}_\ve\rangle=b\sum_{i=1}^n|x_i|
\]
for some $a>0$, $b<\infty$.
Hence for $x\in (0,\infty)^n$
\[
\sup_{\tau\in [0,1]}\left|\frac{\partial}{\partial \tau}f(x,\tau)\right|
\leq g(x):=M\pi^{-n/2}\frac{1}{\sqrt{x_1\cdots x_n}}\left(1+b\sum_{i=1}^n|x_i|\right)e^{-a\sum_{i=1}^n|x_i|}.
\]
Since $g(x)\in L_1((0,\infty)^n$ and $e^{-\sum_{i=1}^n\lambda_i x_i}\geq 1$ the statement easily follows by the Lebesgue
dominated convergence theorem.
\end{proof}

Let for $\alpha>0$, 
\[
g_{\alpha}(x,y):=e^{-x-y}\sum_{k=0}^{\infty}\frac{x^{k+\alpha-1}}{\Gamma(k+\alpha)}\frac{y^k}{k!}\quad
x>0,y\geq 0.
\]
For $\mu,\alpha_1,\ldots,\alpha_n>0$ and a random vector $Y=(Y_1,\ldots,Y_n)$ such that $\Pr(Y_i\geq 0)=1$ we set
\[
h_{\alpha_1,\ldots,\alpha_n,\mu,Y}(x_1,\ldots,x_n)
:=\Ex\left[\prod_{i=1}^n \frac{1}{\mu}g_{\alpha_i}\left(\frac{x_i}{\mu},Y_i\right)\right],\quad
x_1,\ldots,x_n>0.
\] 

\begin{lem}
\label{lem:proph}
Let $\mu>0$ and $Y$ be a random $n$-dimensional vector with nonnegative coordinates. For 
$\alpha=(\alpha_1,\ldots,\alpha_n)\in (0,\infty)^n$ set
$h_{\alpha}:=f_{\alpha_1,\ldots,\alpha_n,\mu,Y}$.\\
i) For any $\alpha\in (0,\infty)^n$, $h_{\alpha}\geq 0$ and $\int_{(0,\infty)^n}h_\alpha(x)dx=1$.\\
ii) If $\alpha\in (0,\infty)^n$ and $\alpha_i>1$ then $\lim_{x_i\rightarrow 0+}h_\alpha(x)=0$,
$\frac{\partial}{\partial x_i}h_\alpha(x)$ exists and
\[
\frac{\partial}{\partial x_i}h_\alpha(x)=h_{\alpha-e_i}-h_{\alpha}.
\]
iii) If $\alpha\in (1,\infty)^n$ then for any $J\subset [n]$,
$\frac{\partial^|J|}{\partial x_{J}}h_\alpha(x)$ exists and belongs to $L_1((0,\infty)^n)$. Moreover
for $\lambda_1,\ldots,\lambda_n\geq 0$,
\[
\int_{(0,\infty)^n}e^{-\sum_{i=1}^n\lambda_i x_i}\frac{\partial^{|J|}}{\partial x_{J}}h_\alpha(x)dx
=\prod_{i\in J}\lambda_{i}\int_{(0,\infty)^n}e^{-\sum_{i=1}^n\lambda_i x_i}h_\alpha(x)dx.
\]
\end{lem}

\begin{proof}
i) Obviously $h_{\alpha}\in [0,\infty]$. We have for any $y\geq 0$ and $\alpha>0$,
\[
\int_{0}^\infty\frac{1}{\mu}g_\alpha\left(\frac{x}{\mu},y\right)dx
=\int_{0}^\infty g_\alpha(x,y)dx=1.
\]
Hence by the Fubini theorem,
\[
\int_{(0,\infty)^n}h_\alpha(x)dx=
\Ex\prod_{j=1}^k\int_{0}^\infty\frac{1}{\mu}g_{\alpha_i}\left(\frac{x_i}{\mu},Y_i\right)dx_i=1.
\]

ii) It is well known that $\Gamma(x)$ is decreasing on $(0,x_0]$ and increasing on $[x_0,\infty)$, where $1<x_0<2$ and
$\Gamma(x_0)>1/2$. Therefore for $k=1,\ldots$ and $\alpha>0$, $\Gamma(k+\alpha)\geq \frac{1}{2} \Gamma(k)=\frac{1}{2} (k-1)!$ and
\[
g_{\alpha}(x,y)\leq e^{-x}\sum_{k=0}^{\infty}\frac{x^{k+\alpha-1}}{\Gamma(k+\alpha)}
\leq 2\left(x^{\alpha-1}e^{-x}+x^{\alpha}\sum_{k=1}^{\infty}\frac{x^{k-1}}{(k-1)!}e^{-x}\right)
=2x^{\alpha-1}(e^{-x}+x).
\]  
This implies that for $\alpha>0$ and $0<a<b<\infty$, $g_{\alpha}(x,y)\leq C(\alpha,a,b)<\infty$ for 
$x\in (a,b)$ and $y\geq 0$.
Moreover,
\[
h_{\alpha}(x)\leq \left(\frac{2}{\mu}\right)^n\prod_{i=1}^n \left(\frac{x_i}{\mu}\right)^{\alpha_i-1}\left(1+\frac{x_i}{\mu}\right).
\]
In particular $\lim_{x_i\rightarrow 0+}h_\alpha(x)=0$ if $\alpha_i\geq 1$.
Observe that for $\alpha>1$, $\frac{\partial}{\partial x}g_\alpha=g_{\alpha-1}-g_{\alpha}$.
Standard application of the Lebegue dominated convergence theorem concludes the proof of part ii).

iii) By ii) we get
\[
\frac{\partial^{|J|}}{\partial x_{J}}h_\alpha
=\sum_{\delta\in \{0,1\}^J}(-1)^{|J|-\sum_{i\in J}\delta_i}f_{\alpha-\sum_{i\in J}\delta_i e_{i}}\in L_1((0,\infty)^n).
\]
Moreover $\lim_{x_j\rightarrow 0+}\frac{\partial^{|J|}}{\partial x_{J}}h_\alpha(x)=0$
for $j\notin J$. We finish the proof by induction on $|J|$ using integration by parts.
\end{proof}

Let $C$ be a strictly positively defined symmetric $n\times n$ matrix. 
Then there exists $\mu>0$ such that $C-\mu I_n$ is positively defined, so 
$C=\mu I_n+AA^T$ for some $A\in M_{n\times n}$. 
Let $(g_{j}^{(l)})_{j\leq n,l\leq k}$ be i.i.d. $\mathcal{N}(0,1)$ r.v's.
Set 
\[
Y_{i}=\frac{1}{2\mu}\sum_{l=1}^k\sum_{j,j'\leq n} g_{j}^{(l)}g_{j'}^{(l)}a_{i,j}a_{i,j'}=
\sum_{l=1}^k\left(\sum_{j=1}^n \frac{1}{\sqrt{2\mu}}g_j^{(l)}a_{i,j}\right)^2,\quad
1\leq i\leq n
\]
and 
\begin{equation}
\label{eq:defhkC}
h_{k,C}:=h_{\frac{k}{2},\ldots,\frac{k}{2},\mu,Y}.
\end{equation}

\begin{lem}
\label{lem:LaplhkC}
For any $\lambda_1,\ldots,\lambda_n\geq 0$ we have
\[
\int_{(0,\infty)^n}e^{-\sum_{i=1}^n\lambda_i x_i}h_{k,C}(x)=|I_n+\Lambda C|^{-\frac{k}{2}},
\] 
where $\Lambda=\diag(\lambda_1,\ldots,\lambda_n)$.
\end{lem}
 
\begin{proof}
We have for any $\alpha,\mu>0$ and $\lambda,y\geq 0$
\begin{align*}
\int_0^\infty \frac{1}{\mu} e^{-\lambda x}g_{\alpha}\left(\frac{x}{\mu},y\right)dx
&=e^{-y}\sum_{k=0}^{\infty}\frac{y^k}{k!\Gamma(k+\alpha)}\int_{0}^{\infty}e^{-(\lambda+\frac{1}{\mu})x}\frac{x^{k+\alpha-1}}{\mu^{k+\alpha}}dx
\\
&=e^{-y}\sum_{k=0}^{\infty}\frac{y^k}{k!(1+\mu\lambda)^{k+\alpha}}=(1+\mu\lambda)^{-\alpha}
e^{-\frac{\mu\lambda}{1+\mu\lambda}y}.
\end{align*}
By the  Fubini theorem we have 
\begin{align*}
\int_{(0,\infty)^n}e^{-\sum_{i=1}^n\lambda_i x_i}h_{k,C}(x)dx
&=
\Ex \prod_{i=1}^n\int_{0}^\infty e^{-\lambda_i x_i}\frac{1}{\mu}g_{k/2}\left(\frac{x_i}{\mu},Y_i\right)dx_i
\\
&=|I_n+\mu\Lambda|^{-\frac{k}{2}}\Ex e^{-\sum_{i=1}^n\frac{\mu\lambda_i}{1+\mu\lambda_i}Y_i}.
\end{align*}
Observe that $Y_i=\sum_{l=1}^k (X_i^{(l)})^2$, where $X^{(l)}:=(X_i^{(l)})_{i\leq n}$ are independent
${\mathcal N}(0,\frac{1}{2\mu}AA^T)$. Therefore by Lemma \ref{lem:Lap} we have
\[
\int_{(0,\infty)^n}e^{-\sum_{i=1}^n\lambda_i x_i}h_{k,C}(x)dx
=|I_n+\mu\Lambda|^{-\frac{k}{2}}|I+2\mu\Lambda(I+\mu\Lambda)^{-1}\frac{1}{2\mu}AA^T|^{-\frac{k}{2}}
=|I_n+C|^{-\frac{k}{2}}.
\]

\end{proof}

\noindent
Institute of Mathematics\\
University of Warsaw\\
Banacha 2\\
02-097 Warszawa, Poland\\
{\tt rlatala@mimuw.edu.pl, ddmatlak@gmail.com}

\end{document}